\documentclass[11pt,leqno]{amsart}
\usepackage{amsmath}
\usepackage{amsfonts}
\usepackage{amssymb}
\usepackage{mathtools}
\usepackage{graphicx}
\usepackage{color,esint,MnSymbol}
\usepackage{hyperref}
\usepackage{xcolor}
\newtheorem{theorem}{Theorem}[section]
\newtheorem*{theorem*}{Theorem}
\newtheorem{lemma}[theorem]{Lemma}

\newtheorem{proposition}[theorem]{Proposition}

\theoremstyle{definition}
\newtheorem{definition}[theorem]{Definition}

\newtheorem{problem}[theorem]{Problem}

\newcommand{\R}{\mathbb{R}}

\def\Ric{\text{Ric}}

\def\R{\mathbb{R}}

\def\Ric{\operatorname{Ric}}

\numberwithin{equation}{section}
\makeatletter
\newcommand*\owedge{\mathpalette\@owedge\relax}
\newcommand*\@owedge[1]{%
  \mathbin{%
    \ooalign{%
      $#1\m@th\bigcirc$\cr
      \hidewidth$#1\m@th\wedge$\hidewidth\cr
    }%
  }%
}
\makeatother

\begin{document}

\title[Sectional Curvature and Isotropic Curvature]{Sectional Curvature, Isotropic Curvature, and Yau's Pinching Problem}

\author[Xiaolong Li]{Xiaolong Li}\thanks{The author's research is partially supported by NSF-DMS \#2405257 and NSF-DMS \#2553660, and a start-up grant at Auburn University}
\address{Department of Mathematics and Statistics, Auburn University, Auburn, AL, 36849, USA}
\email{xil0005@auburn.edu}



\begin{abstract}
We prove that if a closed Riemannian manifold $(M^n,g)$ has finite fundamental group and satisfies the curvature condition
\begin{equation*}
    R_{1313} +R_{1414} +R_{2323} + R_{2424} > \tfrac{1}{2}\left(R_{1212} + R_{3434}\right)
\end{equation*}
for all orthonormal four-frame $\{e_1, e_2, e_3, e_4\} \subset T_pM$, then the universal cover of $M$ is homeomorphic to the $n$-sphere. This generalizes the famous sphere theorem under the stronger condition of $\frac{1}{4}$-pinched sectional curvature.  
As an application, we provide a partial answer to a pinching problem proposed by Yau in 1990.
\end{abstract}

\maketitle

\section{Introduction}

The celebrated sphere theorem due to Berger \cite{Berger60} and Klingenberg \cite{Klingenberg61} asserts that if $(M^n,g)$ is a complete, simply connected, $n$-dimensional Riemannian manifold whose sectional curvature $K$ satisfies
\begin{equation*}
    1 < K (\sigma) \leq 4
\end{equation*}
for all two-planes $\sigma \subset T_pM$, then $M$ is homeomorphic to the $n$-sphere. In 2009, Brendle and Schoen \cite{BS09} utilized the Ricci flow introduced by Hamilton \cite{Hamilton82}, together with the novel construction of invariant cones under Ricci flow by B\"ohm and Wilking \cite{BW08}, to upgrade the homeomorphism to diffeomorphism and weaken the pinching condition from global to pointwise, whereas the $n=4$ case was done much earlier by Chen \cite{Chen91} based on Hamilton's convergence result for Ricci flow on four-manifolds \cite{Hamilton86}. More precisely, it is proved that if a closed Riemannian manifold $(M^n,g)$ has $\frac{1}{4}$-pinched sectional curvature in the sense that 
\begin{equation*}
    K(\sigma_1) > \frac{1}{4} K(\sigma_2)
\end{equation*}
for all two-planes $\sigma_1, \sigma_2 \subset T_pM$, then $M$ is diffeomorphic to a spherical space form. Shortly after, the $\frac{1}{4}$-pinched sectional curvature condition was relaxed to $\frac{1}{4}$-pinched flag curvature, which means $K(\sigma_1) > \frac{1}{4} K(\sigma_2)$ for all two-planes $\sigma_1, \sigma_2 \subset T_pM$ that intersect in a line, by Andrews and Nguyen \cite{AN09} for $n=4$ and Ni and Wilking \cite{NW10} for $n\geq 5$. 
For more sphere theorems under various positivity conditions on curvature, we refer the reader to \cite{MM88}, \cite{BW08}, \cite{Brendle08}, \cite{CGT23}, \cite{Li24, Li22JGA, Li24New}, \cite{PW21}, \cite{NPW22}, and the references therein. 

The main purpose of this article is to prove the following sphere theorem under a much weaker pinching condition on the sectional curvature. 
\begin{theorem}\label{thm sphere main}
Let $(M^n,g)$ be a closed Riemannian manifold of dimension $n\geq 4$. Suppose that $(M,g)$ has finite fundamental group and satisfies 
\begin{equation}\label{eq main condition}
    R_{1313} +R_{1414} +R_{2323} + R_{2424} > \frac{1}{2}\left(R_{1212} + R_{3434}\right)
\end{equation}
for all orthonormal four-frame $\{e_1, e_2, e_3, e_4\} \subset T_pM$. Then the universal cover of $M$ is homeomorphic (diffeomorphic if $n=4$ or $n\geq 12$) to the $n$-sphere. 
\end{theorem}

Clearly, condition \eqref{eq main condition} is implied by $\frac{1}{4}$-pinched sectional or flag curvature (see Lemma \ref{lemma 1/4}). Thus, Theorem \ref{thm sphere main} generalizes the above-mentioned sphere theorems of Berger \cite{Berger60} and Klingenberg \cite{Klingenberg61}, Brendle and Schoen \cite{BS08}, and Ni and Wilking \cite{NW10}. We would like to point out that \eqref{eq main condition} is, in a certain sense, much weaker than $\frac{1}{4}$-pinched sectional or flag curvature, as it does not even imply positive Ricci curvature. This can be seen by considering the manifold $\mathbb{S}^{n-1} \times \mathbb{S}^1$, with the standard product metric, which satisfies \eqref{eq main condition} but does not have positive Ricci curvature.

A key step in the proofs of the $\frac{1}{4}$-pinched differentiable sphere theorems is to show that $\frac{1}{4}$-pinched sectional or flag curvature implies that $M\times \R^2$ has positive isotropic curvature, a condition that is equivalent to positive complex sectional curvature as discovered in \cite{NWolfson07}. Then a metric with positive complex sectional curvature on a closed manifold evolves under the normalized Ricci flow to a metric with constant sectional curvature, as shown by Brendle and Schoen \cite{BS09}. 
On the contrary, condition \eqref{eq main condition} does not imply positive complex sectional curvature, as it does not even imply positive Ricci curvature. 
Instead, our proof of Theorem \ref{thm sphere main} relies on the key observation that \eqref{eq main condition} implies the positive isotropic curvature condition introduced by Micallef and Moore \cite{MM88}. 

\begin{theorem}\label{thm imply PIC}
Let $R$ be an algebraic curvature tensor on a Euclidean vector space $V$ of dimension $n\geq 4$. If $R$ satisfies  
\begin{equation*}
    R_{1313} +R_{1414} +R_{2323} + R_{2424} > \frac{1}{2}\left(R_{1212} + R_{3434}\right)
\end{equation*}
for all orthonormal four-frame $\{e_1, e_2, e_3, e_4\} \subset V$, then $R$
has positive isotropic curvature, i.e.,
\begin{equation*}
    R_{1313} +R_{1414} +R_{2323} + R_{2424} -2R_{1234} >0
\end{equation*}
for all orthonormal four-frame $\{e_1, e_2, e_3, e_4\} \subset V$. 
\end{theorem}
 
The finite fundamental group assumption in Theorem \ref{thm sphere main} is only used to pass to the universal cover, to which one can  apply the beautiful sphere theorem of Micallef and Moore \cite{MM88} stating that a closed, simply connected Riemannian manifold with positive isotropic curvature is homeomorphic to the $n$-sphere. We also point out that the homeomorphism can be improved to diffeomorphism for $n=4$ using Hamilton's classification of closed four-manifolds with positive isotropic curvature \cite{Hamilton97} (see also \cite{CTZ12}) and for $n\geq 12$ with Brendle's work \cite{Brendle19} (see also \cite[Corollary 1.3]{Huang23}). Dimensions $9\leq n \leq 11$ seem reachable due to a recent preprint of Chen \cite{Chen24}, while the $5\leq n \leq 8$ cases remain completely open.

Theorem \ref{thm imply PIC} establishes a novel connection between sectional curvature and isotropic curvature, and it has an application to Yau's pinching problem that we explain now. In 1990, Yau asked in his ``Open problems in geometry" \cite[Problem 12, page 4]{Yau93} (see also \cite[page 369]{SYbook}): 
\textit{``The famous pinching problem says that on a compact simply connected manifold if $K_{\min} > \frac{1}{4}K_{\max}>0$, then the manifold is homeomorphic to a sphere. If we replace $K_{\max}$ by normalized scalar curvature, can we deduce similar pinching theorems?" } 
Here, $K_{\min}$ and $K_{\max}$ denote the minimum and maximum of the sectional curvature, respectively, and the normalized scalar curvature $S_0$ is given by $S_0=\frac{1}{n(n-1)}S$, where $S$ denotes the scalar curvature.  
Yau's pinching problem can be restated as follows (see \cite[page 525]{GX12}). 
\begin{problem}[Yau \cite{Yau93}]\label{conjecture Yau}
Is a closed, simply connected Riemannian manifold $(M^n,g)$ of dimension $n \geq 4$ satisfying
\begin{equation*}
    K_{\min} > \frac{n-1}{n+2}S_0
\end{equation*}
homeomorphic to the $n$-sphere. 
\end{problem}

The best result up to date for Problem \ref{conjecture Yau} is due to Gu and Xu \cite{GX12}, who gave an affirmative answer under the stronger condition 
$K_{\min} > \frac{n(n-1)}{n^2-n+6}S_0$.
In addition, Problem \ref{conjecture Yau} was answered affirmatively in dimension $4$ by Costa and Ribeiro \cite{CR14} under a weaker condition, and for Einstein manifolds in all dimensions by Xu and Gu \cite{XG14}. The constant $\frac{n-1}{n+2}$ is the best possible in view of the complex projective space $\mathbb{CP}^{\frac{n}{2}}$ with the Fubini-Study metric $g_{FS}$.

Our next result provides a partial answer to Yau's pinching problem.
\begin{theorem}\label{thm yau}
Let $(M^n,g)$ be a closed Riemannian manifold of dimension $n\geq 4$. Suppose that $(M^n,g)$ has finite fundamental group and satisfies 
\begin{equation}\label{eq Yau}
    R_{1313} +R_{1414} +R_{2323} + R_{2424} > \frac{4n(n-1)}{n^2-n+12}S_0
\end{equation}
for all orthonormal four-frame $\{e_1, e_2, e_3, e_4\} \subset T_pM$. Then the universal cover of $M$ is homeomorphic to the $n$-sphere. In particular, a closed, simply connected Riemannian manifold satisfying
\begin{equation*}\label{eq Yau stronger}
    K_{\min} > \frac{n(n-1)}{n^2-n+12}S_0
\end{equation*}
is homeomorphic to the $n$-sphere.
\end{theorem}

Theorem \ref{thm yau} improves the result of Gu and Xu \cite{GX12} in two aspects: it improves the pinching constant in front of $S_0$ and relaxes $K_{\min}$ to the minimum of the sum of four sectional curvatures of the form $R_{1313} +R_{1414} +R_{2323} + R_{2424}$. 
In dimension four, Theorem \ref{thm yau} also recovers the result of Costa and Ribeiro \cite{CR14} assuming $K^\perp_{\max} < 2S_0$, where $K^\perp_{\max}$ denotes the maximum of biorthogonal (sectional) curvature. Their key observation is that $K^\perp_{\max} < 2S_0$ implies positive isotropic curvature in dimension four. 
Indeed, the three conditions $K^\perp_{\max} < 2S_0$, \eqref{eq Yau}, and \eqref{eq main condition} are all equivalent when $n=4$, due to the identity
\begin{equation}\label{identity 4D}
    6S_0 =R_{1212}+R_{3434}+R_{1313}+R_{1414}+R_{2323}+R_{2424}.
\end{equation}

It is also natural to replace $K_{\min}$, instead of $K_{\max}$, in the $\frac{1}{4}$-pinched sectional curvature condition by an appropriate multiple of $S_0$ and ask the same question as in Problem \ref{conjecture Yau}. 
Indeed, such a conjecture was formulated by Gu and Xu \cite[Conjecture 2 on page 526]{GX12}, and they proved a partial result stating that a closed Riemannian manifold satisfying 
$K_{\max} < \frac{n(n-1)}{n(n-1)-\frac{12}{5}}S_0$ is diffeomorphic to a spherical space form. Here, we obtain an improvement of their result. 
\begin{theorem}\label{thm Yau K_max}
Let $(M^n,g)$ be a closed Riemannian manifold of dimension $n\geq 4$. Suppose that $(M,g)$ has finite fundamental group and satisfies     
\begin{equation}\label{eq Yau upper}
    R_{1212}+R_{3434} < 2\gamma_n S_0,
\end{equation}
for all orthonormal four-frame $\{e_1,e_2,e_3,e_4\} \subset T_pM$,
where 
\begin{equation}\label{eq gamma_n def}
    \gamma_n =\begin{cases}
        \frac{n(n-1)}{n^2-n-6} & \text{if }  n \geq 6 \text{ or } n=4; \\
        \frac{20}{17} & \text{if } n=5.
    \end{cases}
\end{equation}
Then the universal cover of $M$ is homeomorphic to the $n$-sphere. In particular, a closed, simply connected Riemannian manifold satisfying 
$$K_{\max} < \gamma_n S_0$$ is homeomorphic to the $n$-sphere. 
\end{theorem}

Theorems \ref{thm yau} and \ref{thm Yau K_max} will be derived as consequences of Theorem \ref{thm sphere main}, in the sense that we obtain them by showing that \eqref{eq main condition} is implied by either \eqref{eq Yau} or \eqref{eq Yau upper}. 
As in Theorem \ref{thm sphere main}, the homeomorphism in Theorems \ref{thm yau} and \ref{thm Yau K_max} can be upgraded to diffeomorphism if $n=4$ or $n\geq 12$. For four-manifolds, topological and geometric rigidity results were obtained by Cao and Tran \cite{CT22} under a variety of curvature conditions.

Sphere theorems often have corresponding rigidity results. Brendle and Schoen \cite{BS08} proved that a closed Riemannian manifold with weakly $\frac{1}{4}$-pinched sectional curvature, in the sense that $K(\sigma_1)\geq \frac{1}{4}K(\sigma_2) \geq 0$ for all $\sigma_1, \sigma_2 \subset T_pM$, is either diffeomorphic to a spherical space form or isometric to a rank one compact symmetric space. Ni and Wilking \cite{NW10} showed the same conclusion for manifolds with weakly $\frac{1}{4}$-pinched flag curvature. 
Here, we prove a rigidity result corresponding to Theorem \ref{thm sphere main}. 
\begin{theorem}\label{thm rigidity}
Let $(M^n,g)$ be a closed Riemannian manifold of dimension $n\geq 4$. Suppose that $(M^n,g)$ has finite fundamental group and satisfies 
\begin{equation}\label{eq main condition weakly}
    R_{1313} +R_{1414} +R_{2323} + R_{2424} \geq \frac{1}{2}\left(R_{1212} + R_{3434}\right)
\end{equation}
for all orthonormal four-frame $\{e_1, e_2, e_3, e_4\} \subset T_pM$. 
Then one of the following statements holds: 
\begin{enumerate}
    \item $(M,g)$ is flat;
    \item the universal cover of $M$ is homeomorphic to the $n$-sphere;
    \item $n=2m$ and the universal cover of $M$ is a K\"ahler manifold biholomorphic to $\mathbb{CP}^m$;
    \item the universal cover of $(M,g)$ is isometric to an irreducible compact symmetric space. 
\end{enumerate}
\end{theorem}

The key ingredients in the proof of Theorem \ref{thm rigidity} include Lemma \ref{lemma non-splitting} that prevents certain splitting, the structure theorem of reducible manifolds with nonnegative isotropic curvature in \cite{MW93}, and the classification of simply connected irreducible manifolds with nonnegative isotropic curvature in \cite[Theorem 9.30]{Brendle10book}.

This article is organized as follows. In Section 2, we collect some preliminaries and establish some basic properties related to condition \eqref{eq main condition}. In Section 3, we prove Theorem \ref{thm imply PIC}, namely \eqref{eq main condition} implies positive isotropic curvature, and then derive Theorem \ref{thm sphere main}. In Section 4, we present the proofs of Theorems \ref{thm yau} and \ref{thm Yau K_max}. In Section 5, we prove Theorem \ref{thm rigidity}.

\section{Curvature Conditions}

In this section, we collect some preliminaries and prove several elementary properties and identities that will be used in subsequent sections. 
Throughout this paper, $(V,g)$ is a Euclidean vector space of dimension $n\geq 4$. 

\subsection{Preliminaries}
Let $\wedge^2 V$ denote the space of two-forms on $V$. Denote by $S^2_B(\wedge^2 V)$ the space of algebraic curvature operators (or tensors) on $V$, that is to say, $S^2_B(\wedge^2 V)$ consists of all symmetric operators $R:\wedge^2 V \to \wedge^2 V$ satisfying the symmetries 
\begin{equation*}
    R(X,Y,Z,W)=-R(Y,X,Z,W)=-R(X,Y,W,Z)=R(Z,W,X,Y)
\end{equation*}
and the first Bianchi identity
\begin{equation*}
    R(X,Y,Z,W)+R(X,Z,W,Y)+R(X,W,Y,Z)=0
\end{equation*}
for all $X, Y, Z, W \in V$. 

Let $R \in S^2_B(\wedge^2 V)$. The sectional curvature $K$ of $R$ of a two-plane $\sigma \subset V$ is defined as
$$K(\sigma) =R_{1212},$$
where $\{e_1,e_2\}$ is any orthonormal basis of $\sigma$. Here and in the rest of this paper, we use the abbreviation
\begin{equation*}
    R_{ijkl}=R(e_i,e_j,e_k,e_l)
\end{equation*}
when $\{e_i\}_{i=1}^n$ is an orthonormal basis of $V$. 
The Ricci curvature of $R$ is given by
\begin{equation*}
    R_{ij}=\Ric(e_i,e_j) = \sum_{k=1}^n R(e_i,e_k,e_j,e_k)
\end{equation*}
and the scalar curvature $S$ of $R$ is given by 
\begin{equation*}
S=\sum_{i=1}^n R_{ii} = 2\sum_{1\leq i <j \leq n} R_{ijij}.
\end{equation*}
The normalized scalar curvature $S_0$ is given by
\begin{equation*}
    S_0=\frac{S}{n(n-1)}. 
\end{equation*}

\subsection{Basic properties}
\begin{lemma}\label{lemma 1/4}
Suppose $R \in S^2_B(\wedge^2 V)$ has $\frac{1}{4}$-pinched sectional curvature or $\frac{1}{4}$-pinched flag curvature, then 
\begin{equation*}\label{eq 3.1}
    R_{1313} +R_{1414} +R_{2323} + R_{2424} > \frac{1}{2}\left(R_{1212} + R_{3434}\right)
\end{equation*}
for all orthonormal four-frame $\{e_1, e_2, e_3, e_4\} \subset V$.
\end{lemma}
\begin{proof}
Let $\{e_1, e_2, e_3, e_4\}$ be an arbitrary orthonormal four-frame in $V$. If $R$ has $\frac{1}{4}$-pinched sectional curvature, then
\begin{equation*}
    R_{1313} +R_{1414} +R_{2323} + R_{2424} \geq 4K_{\min} > K_{\max} \geq \tfrac{1}{2}\left(R_{1212} + R_{3434}\right). 
\end{equation*}

If $R$ has $\frac{1}{4}$-pinched flag curvature, then we have $R_{ijij} > \frac{1}{4} R_{ikik}$, whenever $i,j,k$ are mutually distinct. In particular, we have
\begin{align*}
    R_{1313} & > \tfrac{1}{4} R_{1212}, \\
    R_{1414} &> \tfrac{1}{4} R_{3434}, \\
    R_{2323} &> \tfrac{1}{4} R_{1212}, \\
    R_{2424} & > \tfrac{1}{4} R_{3434}. 
\end{align*}
Adding these four inequalities together produces 
\begin{equation*}
    R_{1313} +R_{1414} +R_{2323} + R_{2424} > \tfrac{1}{2}\left(R_{1212} + R_{3434}\right). 
\end{equation*}
The proof is complete. 
\end{proof}

Next, we show that positive scalar curvature is implied by a family of curvature conditions including \eqref{eq main condition}. 
\begin{lemma}\label{lemma 2.1}
Let $\gamma <2$ and $R \in S^2_B(\wedge^2 V)$. Suppose that 
\begin{equation}\label{eq 2.1}
R_{1313} +R_{1414} +R_{2323} + R_{2424} > \gamma \left(R_{1212} + R_{3434}\right) 
\end{equation}
for all orthonormal four-frame $\{e_1, e_2, e_3, e_4\} \subset V$. Then $R$ has positive scalar curvature. If equality is allowed in \eqref{eq 2.1}, then $R$ has nonnegative scalar curvature.  
\end{lemma}

We point out that the $\gamma=0$ case covers a well-known result of Micallef and Wang \cite[Page 659]{MW93} stating that positive isotropic curvature implies positive scalar curvature.

To prove Lemma \ref{lemma 2.1}, we first establish two identities that will also be used in the proofs of Propositions \ref{prop Yau} and \ref{prop Yau K_max} later.

\begin{lemma}
Let $R \in S^2_B(\wedge^2 V)$ and $\{e_i\}_{i=1}^n$ be an orthonormal basis of $V$. Then 
\begin{eqnarray}\label{identity 2.2}
&& \sum_{1 \leq i \neq j \neq k \neq l \leq n} \left(R_{ikik}+R_{ilil}+R_{jkjk}+R_{jljl}\right) \\ \nonumber
&=& 8(n-3)(n-2) \sum_{1 \leq i < j \leq n} R_{ijij} 
\end{eqnarray}
and 
\begin{eqnarray}\label{identity 2.3}
   && \sum_{1 \leq i \neq j \neq k \neq l \leq n} \left(R_{ijij}+R_{klkl}\right) \\ \nonumber
&=& 4(n-3)(n-2) \sum_{1 \leq i < j \leq n} R_{ijij}  .
\end{eqnarray}
Here, $i\neq j \neq k \neq l$ in the summation means $i,j,k,l$ are mutually distinct. 
\end{lemma}

\begin{proof}
To establish \eqref{identity 2.2}, we compute that 
\begin{eqnarray*}
&& \sum_{1 \leq i \neq j \neq k \neq l \leq n} \left(R_{ikik}+R_{ilil}+R_{jkjk}+R_{jljl} \right) \\
&=& \sum_{1 \leq i \neq j \neq k \leq n} \left(R_{ii}+R_{jj} +(n-4)(R_{ikik}+R_{jkjk}) -2R_{ijij} \right) \\
&=& 2(n-3) \sum_{1 \leq i \neq j \leq n}  \left((R_{ii}+R_{jj}) -2 R_{ijij}  \right)\\
&=& 2(n-3) \sum_{1 \leq i \leq n} \left( (n-4)R_{ii} + S \right) \\
&=& 4(n-3)(n-2) S \\
&=& 8(n-3)(n-2) \sum_{1 \leq i < j \leq n} R_{ijij} . 
\end{eqnarray*}

Similarly, one verifies \eqref{identity 2.3} as follows:
\begin{eqnarray*}
&& \sum_{1 \leq i \neq j \neq k \neq l \leq n} \left(R_{ijij}+R_{klkl} \right) \\
&=& \sum_{1 \leq i \neq j \neq k \leq n} \left((n-3)R_{ijij}+R_{kk}-R_{ikik}-R_{jkjk} \right) \\
&=& \sum_{1 \leq i \neq j \leq n}  \left(((n-2)(n-3)+2)R_{ijij}+S -2(R_{ii}+R_{jj})   \right)\\
&=& (n-3)\sum_{1 \leq i \leq n} \left( (n-4)R_{ii} + S \right) \\
&=& 2(n-3)(n-2) S \\
&=& 4(n-3)(n-2) \sum_{1 \leq i < j \leq n} R_{ijij} . 
\end{eqnarray*}

\end{proof}

We now prove Lemma \ref{lemma 2.1}. 
\begin{proof}[Proof of Lemma \ref{lemma 2.1}]
Let $\{e_i\}_{i=1}^n$ be an orthonormal basis of $V$. Since $\{e_i,e_j,e_k,e_l\}$ is an orthonormal four-frame in $V$ for $1\leq i,j,k,l \leq n$ mutually distinct, we have 
\begin{equation*}
    R_{ikik}+R_{ilil}+R_{jkjk}+R_{jljl}-\gamma(R_{ijij}+R_{klkl}) > 0.
\end{equation*}
Using \eqref{identity 2.2} and \eqref{identity 2.3}, one gets
\begin{eqnarray*}\label{identity 2.1}
&& 2(2-\gamma)(n-3)(n-2) S \\ \nonumber
&=& \sum_{1 \leq i \neq j \neq k \neq l \leq n} \left(R_{ikik}+R_{ilil}+R_{jkjk}+R_{jljl}-\gamma(R_{ijij}+R_{klkl}) \right). 
\end{eqnarray*}
It follows that $S>0$ when $\gamma <2$. 
This finishes the proof. 
\end{proof}

\subsection{Complex Projective Spaces}
Let $R^{\mathbb{CP}^m}$ denote the Riemann curvature tensor of the complex projective space $\mathbb{CP}^m$ of complex dimension $m\geq 2$ with the Fubini-Study metric, normalized with constant holomorphic sectional curvature $4$. According to \cite{BK78}, $R^{\mathbb{CP}^m}$ is given by
\begin{eqnarray*}
    R^{\mathbb{CP}^m}(X,Y,Z,W) 
    &=& g(X,Z)g(Y,W)-g(X,W)g(Y,Z) \\
    && + g(JX,Z)g(JY,W)-g(JX,W)g(JY,Z) \\
    && +2g(JX,Y)g(JZ,W)
\end{eqnarray*}
for $X,Y,Z,W \in T_p \mathbb{CP}^m$. 
This implies that, if $X$ and $Y$ are unit vectors with $g(X,Y)=0$, then 
\begin{equation*}
    R^{\mathbb{CP}^m}(X,Y,X,Y)= 1+3 (g(JX,Y))^2.
\end{equation*}
In particular, the sectional curvatures of $R^{\mathbb{CP}^m}$ take values in the interval $[1,4]$. It follows from Lemma \ref{lemma 1/4} that $R^{\mathbb{CP}^m}$ satisfies 
\begin{equation*}
    R^{\mathbb{CP}^m}_{1313} +R^{\mathbb{CP}^m}_{1414}+R^{\mathbb{CP}^m}_{2323}+R^{\mathbb{CP}^m}_{2424} \geq \frac{1}{2}(R^{\mathbb{CP}^m}_{1212}+R^{\mathbb{CP}^m}_{3434})
\end{equation*}
for all orthonormal four-frame $\{e_1,e_2,e_3,e_4\}$. Moreover, one easily sees that for an orthonormal four-frame of the form $\{e_1,e_2=Je_1,e_3,e_4=Je_3\}$, we have the equality
\begin{equation*}
   R^{\mathbb{CP}^m}_{1313} + R^{\mathbb{CP}^m}_{1414} +R^{\mathbb{CP}^m}_{2323} + R^{\mathbb{CP}^m}_{2424} =\frac{1}{2} (R^{\mathbb{CP}^m}_{1212} + R^{\mathbb{CP}^m}_{3434} ).
\end{equation*}
It is this example that inspired the author to formulate the curvature condition \eqref{eq main condition} and investigate it in this paper. 

Similar properties are also satisfied by the quaternion-K\"ahler projective space and the Cayley plane, with their canonical metrics.

\section{Proof of Theorem \ref{thm sphere main}}

We recall the important notion of positive isotropic curvature introduced by Micallef and Moore \cite{MM88}. 
\begin{definition}
$R \in S^2_B(\wedge^2 V)$ is said to have positive isotropic curvature if \begin{equation*}
    R_{1313} +R_{1414} +R_{2323} + R_{2424} -2R_{1234} >0
\end{equation*}
for all orthonormal four-frame $\{e_1, e_2, e_3, e_4\} \subset V$. If $>$ is replaced with $\geq$, then $R$ is said to have nonnegative isotropic curvature. 
\end{definition}

We restate Theorem \ref{thm imply PIC} below and present its proof. 
\begin{proposition}\label{prop main}
Suppose that $R \in S^2_B(\wedge^2 V)$ satisfies 
\begin{equation}\label{eq 3.10}
    R_{1313} +R_{1414} +R_{2323} + R_{2424} > \frac{1}{2}\left(R_{1212} + R_{3434}\right)
\end{equation}
for all orthonormal four-frame $\{e_1, e_2, e_3, e_4\} \subset V$. Then $R$ has positive isotropic curvature. If equality is allowed in \eqref{eq 3.10}, then $R$ has nonnegative isotropic curvature. 
\end{proposition}

\begin{proof}
Let $\{e_1, e_2, e_3, e_4\}$ be an arbitrary orthonormal four-frame in $V$. By assumption, we have 
\begin{equation}\label{eq 3.4}
    4 \left(R_{1313} +R_{1414} +R_{2323} + R_{2424}\right) -2 \left(R_{1212} + R_{3434}\right) >0. 
\end{equation}
Set 
\begin{align*}
    e_1' &= \tfrac{1}{\sqrt{2}} (e_1+e_4), \\
    e_2' &= \tfrac{1}{\sqrt{2}} (e_2-e_3), \\
    e_3' &= \tfrac{1}{\sqrt{2}} (e_1-e_4), \\
    e_4' &= \tfrac{1}{\sqrt{2}} (e_2+e_3).
\end{align*}
One verifies that $\{e_1', e_2', e_3', e_4'\}$ is also an orthonormal four-frame in $V$. The assumption then implies 
\begin{equation}\label{eq 3.5}
    4 \left( R'_{1313} +R'_{1414} +R'_{2323} + R'_{2424} \right) - 2 \left(R'_{1212} + R'_{3434}\right) >0,
\end{equation}
where 
$$R'_{ijkl}:=R(e_i',e_j',e_k',e_l').$$ 
By direct calculations, we have
\begin{align*}
4R'_{1313} &= R(e_1+e_4,e_1-e_4, e_1+e_4,e_1-e_4)= 4R_{1414} , \\
4R'_{2424}&= R(e_2-e_3,e_2+e_3, e_2-e_3,e_2+e_3)= 4R_{2323} ,
\end{align*}
\begin{eqnarray*}
    && 4R'_{1414} +4R'_{2323}  \\
    &=& R(e_1+e_4,e_2+e_3, e_1+e_4,e_2+e_3) \\
    && +R(e_1-e_4,e_2-e_3, e_1-e_4,e_2-e_3) \\
    &=& 2\left(R_{1212}+R_{1313}+R_{2424}+R_{3434}+2R_{1243}+2R_{1342} \right), 
\end{eqnarray*}
and 
\begin{eqnarray*}
    && 4R'_{1212} +4R'_{3434} \\
    &=& R(e_1+e_4,e_2-e_3, e_1+e_4,e_2-e_3) \\
    && +R(e_1-e_4,e_2+e_3, e_1-e_4,e_2+e_3) \\
    &=& 2\left(R_{1212}+R_{1313}+R_{2424}+R_{3434}-2R_{1243}-2R_{1342} \right). 
\end{eqnarray*}
Substituting the above identities into \eqref{eq 3.5} produces
\begin{eqnarray}\label{eq 3.6}
0 &<& 4\left( R_{1414}+R_{2323} \right) + 6R_{1243} +6R_{1342}  \\ \nonumber
&& + R_{1212} +R_{1313} +R_{2424} +R_{3434}.
\end{eqnarray}
We then replace $\{e_1, e_2, e_3, e_4\}$ by $\{e_1, e_2, -e_4, e_3\}$ in the above argument to obtain
\begin{eqnarray}\label{eq 3.7}
0 & < & 4\left( R_{1313}+R_{2424} \right)  -6R_{1234} -6R_{1432} \\ \nonumber
&& +R_{1212} +R_{1414} +R_{2323} +R_{3434}. 
\end{eqnarray}
Finally, adding \eqref{eq 3.4}, \eqref{eq 3.6}, and \eqref{eq 3.7} together yields
\begin{eqnarray*}
    0 &<& 9(R_{1313}+R_{1414}+R_{2323}+R_{2424}) -12 R_{1234} +6 R_{1342} +6R_{1423} \\
    &=& 9(R_{1313}+R_{1414}+R_{2323}+R_{2424}) -18 R_{1234},
\end{eqnarray*}
where, in getting the last equality, we have used the first Bianchi identity
$$R_{1342}+R_{1423} =-R_{1234}.$$ 

Since the orthonormal four-frame $\{e_1, e_2, e_3, e_4\}$ is arbitrary, we conclude that $R$ has positive isotropic curvature. 
The statement for nonnegative isotropic curvature can be proved similarly. 
\end{proof}

Next, we prove Theorem \ref{thm sphere main}. 
\begin{proof}[Proof of Theorem \ref{thm sphere main}]
Since $(M^n,g)$ has finite fundamental group, its universal cover $(\widetilde{M}, \tilde{g})$ is a closed and simply connected manifold. Moreover, $(\widetilde{M}, \tilde{g})$ has positive isotropic curvature, in view of Proposition \ref{prop main}. Invoking the sphere theorem of Micallef and Moore \cite{MM88}, we conclude that $\widetilde{M}$ is homeomorphic to $\mathbb{S}^n$. 

The homeomorphism can be improved to diffeomorphism for $n=4$ using Hamilton's classification of four-manifolds with positive isotropic curvature \cite{Hamilton97} (see also \cite{CTZ12}) and for $n\geq 12$ using Brendle's recent work \cite{Brendle19} (see also \cite{Huang23}). 
\end{proof}

\section{Proofs of Theorems \ref{thm yau} and \ref{thm Yau K_max}}

Theorem \ref{thm yau} follows from Theorem \ref{thm sphere main} and the following proposition. 
\begin{proposition}\label{prop Yau}
Suppose $R \in S^2_B(\wedge^2 V)$ satisfies
\begin{equation*}
    R_{1313}+R_{1414}+R_{2323}+R_{2424} > \frac{4n(n-1)}{n^2-n+12} S_0
\end{equation*}
for all orthonormal four-frame $\{e_1, e_2, e_3, e_4\} \subset V$. 
Then 
\begin{equation*}
 R_{1313} +R_{1414} +R_{2323} + R_{2424} > \frac{1}{2} \left(R_{1212}+R_{3434}\right)  
\end{equation*}
for all orthonormal four-frame $\{e_1, e_2, e_3, e_4\} \subset V$. 
\end{proposition}

\begin{proof}[Proof of Proposition \ref{prop Yau}]
We may assume $n\geq 5$, since the assumption and the conclusion coincide when $n=4$ due to the identity \eqref{identity 4D}. 

Let $\{e_1,e_2,e_3,e_4\}$ be an arbitrary orthonormal four-frame in $V$ and we extend it to an orthonormal basis $\{e_i\}_{i=1}^n$ of $V$. 
For $1\leq i,j,k,l \leq n$ that are mutually distinct, $\{e_i,e_j,e_k,e_l\}$ is an orthonormal four-frame in $V$, and the assumption implies 
\begin{equation}\label{eq yau S2}
    R_{ikik}+R_{ilil}+R_{jkjk}+R_{jljl} > 4\eta_n S_0,
\end{equation}
where 
$$\eta_n=\frac{n(n-1)}{n^2-n+12}.$$
To prove Proposition \ref{prop Yau}, it suffices to establish the upper bound 
\begin{equation}\label{eq upper bound needed}
    R_{1212}+R_{3434} < 8\eta_n S_0,
\end{equation} 
in view of 
$$R_{1313}+R_{1414}+R_{2323}+R_{2424} \geq 4\eta_n S_0.$$  
To prove \eqref{eq upper bound needed}, we shall make use of the identity
\begin{equation*}\label{eq scalar}
    \frac{1}{2}n(n-1)S_0 =\sum_{1\leq i < j \leq n} R_{ijij}. 
\end{equation*}

We begin with the $n=5$ case. Notice that
\begin{eqnarray*}
    && 20S_0 -2(R_{1212}+R_{3434}) \\
    &=& (R_{1313} +R_{1414} +R_{5353} +R_{5454}) \\
    &&+(R_{2323} +R_{2424} +R_{5353} +R_{5454})  \\
    &&  +(R_{1313} +R_{1515} +R_{2323} +R_{2525}) \\
    && +(R_{1414} +R_{1515} +R_{2424} +R_{2525}) \\
    & > & 16 \eta_5 S_0,
\end{eqnarray*}
where we have used in the inequality step that each bracket containing four terms is greater than $4\eta_n S_0$ due to \eqref{eq yau S2}. 
It follows that 
\begin{equation*}
    R_{1212}+R_{3434} < 10S_0 -8\eta_5 S_0 =8 \eta_5 S_0,
\end{equation*}
as desired. 

For $n=6$, we observe that 
\begin{eqnarray*}
 &&   15S_0 -(R_{1212}+R_{3434} ) \\ 
&=& \tfrac{3}{4} (R_{1313}+R_{2424}+R_{1414}+R_{2323} ) \\
&& +\tfrac{3}{4} (R_{1515}+R_{1616}+R_{2525}+R_{2626} ) \\
&& +\tfrac{3}{4} (R_{3535}+R_{3636}+R_{4545}+R_{4646} ) \\ 
&&+\tfrac{1}{8} (R_{1313}+R_{1616}+R_{5353}+R_{5656} )  \\
&&+ \tfrac{1}{8} (R_{1414} +R_{1616} +R_{5454} +R_{5656} ) \\
&& + \tfrac{1}{8} (R_{2323} +R_{2626} +R_{5353} +R_{5656} )  \\
&&+ \tfrac{1}{8} (R_{2424} +R_{2626} +R_{5454} +R_{5656} ) \\
&& + \tfrac{1}{8} (R_{3131} +R_{3636} +R_{5151} +R_{5656} )  \\
&& + \tfrac{1}{8} (R_{4141} +R_{4646} +R_{5151} +R_{5656} ) \\
&& + \tfrac{1}{8} (R_{3232} +R_{3636} +R_{5252} +R_{5656} ) \\
&&+ \tfrac{1}{8} (R_{4242} +R_{4646} +R_{5252} +R_{5656} ) \\
&> & 13 \eta_6 S_0.
\end{eqnarray*}
Again, we have used in the last step that each bracket containing four terms is greater than $4\eta_n S_0$ by \eqref{eq yau S2}. 
From this, we deduce 
\begin{equation*}
    R_{1212}+R_{3434} < 15 S_0-13\eta_6 S_0= 8\eta_6 S_0,
\end{equation*}
and finish the $n=6$ case.

When $n=7$, one verifies (a bit tedious but elementary) that 
\begin{eqnarray*}
&& 21S_0 -(R_{1212}+R_{3434}) \\
&=& \frac{1}{24} \sum_{j=3}^4 \sum_{p=5}^7 (R_{1j1j} +R_{1p1p} +R_{2j2j} +R_{2p2p} ) \\
&& +\frac{1}{24} \sum_{i=1}^2 \sum_{p=5}^7 (R_{i3i3} +R_{i4i4} +R_{3p3p} +R_{4p4p} ) \\
&& + \frac{1}{8} \sum_{i=1}^2 \sum_{j=3}^4 \sum_{5 \leq p \neq q \leq 7} (R_{ijij} +R_{iqiq} +R_{jpjp} +R_{pqpq} ) \\
&& + \frac{5}{48} \sum_{i=1}^2 \sum_{j=3}^4 \sum_{5 \leq p < q \leq 7} (R_{ipip} +R_{iqiq} +R_{jpjp} +R_{jqjq} ) \\
\end{eqnarray*}
As before, each bracket containing four terms is greater than $4\eta_n S_0$. Therefore, we have
\begin{equation*}
R_{1212}+R_{3434} < 21 S_0-19\eta_7 S_0 = 8\eta_7 S_0. 
\end{equation*}
This completes the $n=7$ case. 

Finally, we treat all $n\geq 8$. By \eqref{identity 2.2} and \eqref{eq yau S2}, we have 
\begin{eqnarray*}
&& 8(n-6)(n-7)\sum_{5\leq p < q \leq n} R_{pqpq} \\
&=& \sum_{5 \leq i\neq j \neq k \neq l \leq n} (R_{ikik}+R_{ilil}+R_{jkjk}+R_{jljl}) \\
&>& 4(n-7)(n-6)(n-5)(n-4)\eta_n S_0.
\end{eqnarray*}
We also observe that, for any $n\geq 6$, 
\begin{eqnarray*}
&& (n-5) \sum_{i=1}^4\sum_{j=5}^n R_{ijij} \\
&=& \sum_{5 \leq p < q \leq n} (R_{1p1p}+R_{1q1q}+R_{2p2p}+R_{2q2q} ) \\
&& +\sum_{5 \leq p < q \leq n} (R_{3p3p}+R_{3q3q}+R_{4p4p}+R_{4q4q}) \\
&>& 4(n-5)(n-4)\eta_n S_0. 
\end{eqnarray*}
Using the above two estimates, we obtain
\begin{eqnarray*}
   && R_{1212}+R_{3434} \\
   &=& \tfrac{1}{2}n(n-1)S_0 - (R_{1313} +R_{1414} +R_{2323} + R_{2424}) \\
   && -\sum_{i=1}^4\sum_{j=5}^n R_{ijij} -\sum_{5 \leq p < q \leq n} R_{pqpq} \\
    &<& \tfrac{1}{2}n(n-1)S_0 -4\eta_nS_0 -4(n-4)\eta_n S_0 \\
    && -\tfrac{1}{2}(n-5)(n-4)\eta_n S_0\\
    &=& 8\eta_n S_0.
\end{eqnarray*}
This establishes \eqref{eq upper bound needed} for all $n\geq 8$. 

We have completed the proof. 

\end{proof}

Theorem \ref{thm Yau K_max} is a consequence of Theorem \ref{thm sphere main} and the following proposition. 

\begin{proposition}\label{prop Yau K_max}
Suppose $R \in S^2_B(\wedge^2 V)$ satisfies
\begin{equation*}
    R_{1212}+R_{3434}< 2\gamma_n S_0
\end{equation*}
for all orthonormal four-frame $\{e_1, e_2, e_3, e_4\} \subset V$, where $\gamma_n$ is defined in \eqref{eq gamma_n def}. 
Then 
\begin{equation*}
 R_{1313} +R_{1414} +R_{2323} + R_{2424} > \frac{1}{2} \left(R_{1212}+R_{3434}\right)  
\end{equation*}
for all orthonormal four-frame $\{e_1, e_2, e_3, e_4\} \subset V$.  
\end{proposition}

\begin{proof}
The $n=4$ case follows from the identity \eqref{identity 4D}. We assume $n\geq 5$ below.

Let $\{e_1,e_2,e_3,e_4\}$ be an arbitrary orthonormal four-frame in $V$ and we extend it to an orthonormal basis $\{e_i\}_{i=1}^n$ of $V$. The assumption implies that 
\begin{equation}\label{eq 5.05}
    R_{ijij}+R_{klkl} < 2\gamma_n S_0
\end{equation} whenever $1\leq i,j,k,l \leq n$ are mutually distinct.

Using \eqref{identity 2.3} and \eqref{eq 5.05}, we have, for any $n\geq 8$, 
\begin{eqnarray}\label{eq 5.1}
    \sum_{5\leq i < j \leq n} R_{ijij} 
   &=& \frac{1}{4(n-7)(n-6)}\sum_{5\leq i \neq j \neq k \neq l \leq n} (R_{ijij}+R_{klkl}) \\ \nonumber
   &<& \frac{1}{2}(n-5)(n-4)\gamma_n S_0. 
\end{eqnarray}
One verifies that, for $n\geq 6$, we have
\begin{eqnarray*}
    && (n-5)\sum_{i=1}^4 \sum_{j=5}^n R_{ijij} \\
    &=& \sum_{5 \leq p < q \leq n} \left((R_{1p1p}+R_{2q2q})+(R_{3p3p}+R_{4q4q}) \right) \\
    && + \sum_{5 \leq p < q \leq n} \left((R_{1q1q}+R_{2p2p})+(R_{3q3q}+R_{4p4p}) \right) .
\end{eqnarray*}
This, together with \eqref{eq 5.05}, yields
\begin{equation}\label{eq 5.2}
   \sum_{i=1}^4 \sum_{j=5}^n R_{ijij}  <  4(n-4)\gamma_n S_0.
\end{equation}
Therefore, we obtain, when $n\geq 8$, that
\begin{eqnarray*}
   && R_{1313}+R_{1414}+R_{2323}+R_{2424}-\tfrac{1}{2}n(n-1)S_0  \\
&=&  -(R_{1212}+R_{3434}) - \sum_{i=1}^4 \sum_{j=5}^n R_{ijij} - \sum_{5\leq i < j \leq n} R_{ijij} \\
&> & -2\gamma_n S_0 - 4(n-4)\gamma_n S_0 - \tfrac{1}{2}(n-5)(n-4)\gamma_n S_0 \\
&=&  -\frac{1}{2}(n^2-n-8)\gamma_n S_0.
\end{eqnarray*}
In view of \eqref{eq gamma_n def}, this becomes 
\begin{equation*}
    R_{1313}+R_{1414}+R_{2323}+R_{2424} > \gamma_n S_0 > \tfrac{1}{2} (R_{1212}+R_{3434}),
\end{equation*}
as desired. This completes the proof for $n\geq 8$. 

Next, we deal with the cases $5\leq n \leq 7$ separately. 
For $n=7$, we observe that
\begin{eqnarray*}
&& (R_{1313}+R_{1414}+R_{2323}+R_{2424} )-\tfrac{1}{2}(R_{1212}+R_{3434})\\
   &=& 21S_0 -(R_{1212}+R_{5656}) -(R_{3434}+R_{5757})  \\
   && -(R_{1515}+R_{2626}) -(R_{4747}+R_{2525}) \\
   && -(R_{3535}+R_{4646}) -(R_{1717}+R_{4545}) \\
   && -(R_{1616}+R_{3737}) -(R_{3636}+R_{2727}) \\
   && -\tfrac{1}{2}(R_{1212}+R_{6767})-\tfrac{1}{2}(R_{3434}+R_{6767}) \\
   &>& 21 S_0 - 18\gamma_7 S_0 \\
   &=& 0,
\end{eqnarray*}
where we have used \eqref{eq 5.05} to obtain the inequality step. 

In a similar fashion, the $n=6$ case follows from
\begin{eqnarray*}
&& R_{1313}+R_{1414}+R_{2323}+R_{2424} -\tfrac{1}{2}(R_{1212}+R_{3434})\\
   &=& 15S_0 -(R_{1212}+R_{3434})   \\
   && -(R_{1515}+R_{2626}) -(R_{1616}+R_{2525}) \\
   && -(R_{3535}+R_{4646}) -(R_{3636}+R_{4545}) \\
   && -\tfrac{1}{2}(R_{1212}+R_{5656}) -\tfrac{1}{2}(R_{3434}+R_{5656})\\
   &>& 15S_0 - 12\gamma_6 S_0 \\
   &=& 0.
\end{eqnarray*}

Finally, for $n=5$, we notice that
\begin{eqnarray*}
    && R_{1313}+R_{1414}+R_{2323}+R_{2424} \\
   &=& 20S_0 -(R_{1212}+R_{3535})-(R_{1212}+R_{4545}) \\
   &&  -(R_{3434}+R_{1515}) -(R_{3434}+R_{2525}) \\ 
   && -(R_{1313}+R_{2525}) -(R_{1414}+R_{3535}) \\
   && -(R_{2323}+R_{4545})-(R_{2424}+R_{1515}) \\
   &>& 20S_0 - 16\gamma_5 S_0 \\
   &=& \gamma_5 S_0. 
\end{eqnarray*}
Since $R_{1212}+R_{3434} < 2\gamma_5 S_0$, we arrive at 
\begin{equation*}
    R_{1313}+R_{1414}+R_{2323}+R_{2424} > \tfrac{1}{2}(R_{1212}+R_{3434}). 
\end{equation*}
The proof is complete. 

\end{proof}

\section{Rigidity}
In this section, we present the proof of Theorem \ref{thm rigidity}. 
We begin with the following lemma, which prevents non-flat manifolds satisfying \eqref{eq main condition weakly} from certain splitting.  
\begin{lemma}\label{lemma non-splitting}
Suppose that $(M^n,g)=(M_1^{k},g_1) \times (M_2^{n-k}, g_2)$, $n\geq 4$ and $2 \leq k\leq \frac{n}{2}$, satisfies 
\begin{equation*}
    R_{1313} +R_{1414} +R_{2323} + R_{2424} \geq \frac{1}{2}\left(R_{1212} + R_{3434}\right)
\end{equation*}
for all orthonormal four-frame $\{e_1, e_2, e_3, e_4\} \subset T_pM$. If both $(M_1,g_1)$ and $(M_2,g_2)$ have nonnegative scalar curvature somewhere, then $(M^n,g)$ is flat. 
\end{lemma}

\begin{proof}
For $i=1,2$, let $\sigma_i \subset T_{p_i}M_i$ be a two-plane that maximizes the sectional curvature on $(M_i,g_i)$. 
Let $\{e_1,e_2\}$ be an orthonormal basis of $\sigma_1$ and $\{e_3,e_4\}$ be an orthonormal basis of $\sigma_2$. 
For the orthonormal four-frame $\{e_1,e_2,e_3,e_4\} \subset T_{(p_1,p_2)}M$, the curvature assumption implies 
\begin{equation*}
R_{1313} +R_{1414} +R_{2323} + R_{2424} \geq \frac{1}{2}\left(R_{1212} + R_{3434}\right).
\end{equation*}
Due to the product structure, we have $R_{1313}=R_{1414}=R_{2323}=R_{2424}=0$. It follows that $R_{1212}+R_{3434} \leq 0$. 
On the other hand, since both $(M_1,g_1)$ and $(M_2,g_2)$ have nonnegative scalar curvature somewhere, we have $R_{1212} \geq 0$ and $R_{3434}\geq 0$. Hence $R_{1212}=R_{3434}=0$. 
Since $R_{1212}$ and $R_{3434}$ maximize the sectional curvature of $(M_1,g_1)$ and $(M_2,g_2)$, respectively, we conclude that both $(M_1,g_1)$ and $(M_2,g_2)$ have nonpositive sectional curvature, and so does $(M,g)$. 
On the other hand, $(M,g)$ has nonnegative scalar curvature by Lemma \ref{lemma 2.1}. Hence, $(M,g)$ is flat.

\end{proof}

We are ready to prove Theorem \ref{thm rigidity}.
\begin{proof}[Proof of Theorem \ref{thm rigidity}]
Denote by $(\widetilde{M}^n,\tilde{g})$ the Riemannian universal cover of $(M^n,g)$. Since $M$ has finite fundamental group, $\widetilde{M}^n$ is also a closed manifold. 
By Proposition \ref{prop main}, $(\widetilde{M}^n,\tilde{g})$ has nonnegative isotropic curvature. 

We first consider the case that $(\widetilde{M}^n,\tilde{g})$ is irreducible. By the classification of closed, simply connected, irreducible Riemannian manifolds with nonnegative isotropic curvature in \cite[Theorem 9.30]{Brendle10book}, we conclude that  $(\widetilde{M}^n,\tilde{g})$ is either homeomorphic to $\mathbb{S}^n$, or K\"ahler and biholomorphic to $\mathbb{CP}^{\frac{n}{2}}$, or isometric to an irreducible compact symmetric space.

We then consider the case $(\widetilde{M}^n,\tilde{g})$ is reducible and use Lemma \ref{lemma non-splitting} to prove that $(\widetilde{M}^n,\tilde{g})$ must be flat in this case. 
Since the de Rham decomposition of $(\widetilde{M}^n,\tilde{g})$ cannot have Euclidean factors, we have two possibilities by \cite[Theorem 3.1]{MW93}:

Case 1: $(\widetilde{M}^n,\tilde{g})$ is isometric to $(N_1^{n_1}, g_1) \times \cdots \times (N_r^{n_r},g_r)$, where $n=n_1+\cdots +n_r$, and for each $1\leq i \leq r$, either $n_i=2$, $N_i=\mathbb{S}^2$, and $g_i$ has nonnegative Gauss curvature, or $n_i \geq 3$, $N_i$ is compact and irreducible, and $g_i$ has nonnegative Ricci curvature. 
In view of Lemma \ref{lemma non-splitting}, $(\widetilde{M}^n,\tilde{g})$ must be flat. Hence \((M,g)\) is flat.

Case 2: $(\widetilde{M}^n,\tilde{g})$ is isometric to $(\Sigma^2,g_1) \times (N^{n-2}, g_2)$, where $\Sigma$ is a surface with Gauss curvature $\kappa$ negative somewhere, and $(N,g_2)$ is a compact, irreducible Riemannian manifold with nonnegative complex sectional curvature. 
If $\kappa$ is nonnegative somewhere on $\Sigma$, then Lemma \ref{lemma non-splitting} implies that $(\tilde{M},\tilde{g})$ is flat.  
Therefore, \(\Sigma\) must have negative Gauss curvature everywhere. By the
Gauss-Bonnet theorem, \(\chi(\Sigma)<0\). On the other hand, since
\(\Sigma\) is a de Rham factor of the simply connected manifold
\(\widetilde M\), it is itself simply connected. Hence, \(\Sigma\) is closed and diffeomorphic to \(\mathbb{S}^2\), so \(\chi(\Sigma)=2\), a contradiction.

The proof is complete. 

\end{proof}

\section*{Conflict of Interest}
The author states that there is no conflict of interest.

\section*{Data Availability Statement}
Data sharing is not applicable to this article as no datasets were generated or analyzed during the current study.

\bibliographystyle{alpha}
\bibliography{ref}

\end{document}